\documentclass[10pt]{article}

\usepackage{amsmath,amsthm,amssymb}
\usepackage[bookmarksnumbered=true,pagebackref]{hyperref}
\usepackage{tocbibind}
\usepackage{enumitem}
\usepackage{graphicx,float}

\newtheorem{thm}{Theorem}

\newtheorem{lem}{Lemma}
\newtheorem{cor}{Corollary}

\theoremstyle{remark}
\newtheorem{rem}{Remark}
\newtheorem*{acknowledgments}{Acknowledgments}

\theoremstyle{definition}

\newtheoremstyle{notes}
{3pt}
{3pt}
{}
{}
{\bfseries}
{:}
{.5em}
{}
\theoremstyle{notes}
\newtheorem*{keywords}{Keywords}
\newtheorem*{subjclass}{AMS MSC 2010}

\DeclareMathOperator*{\sech}{sech}
\DeclareMathOperator*{\Real}{Re}
\DeclareMathOperator*{\Imaginary}{Im}

\DeclareMathOperator*{\Var}{Var}

\usepackage[draft]{todonotes} 
\usepackage{comment}
\title{Conformal Skorokhod embeddings\\ and related extremal problems}

\author{Phanuel Mariano\footnote{Supported in part by an AMS-Simons Travel Grant 2019-2021}\\
\href{mailto:PMariano@newhaven.edu}{\texttt{{\small PMariano@newhaven.edu}}}
\and
Hugo Panzo\footnote{Supported at the Technion by a Zuckerman Fellowship.}\\ 
\href{mailto:panzo@campus.technion.ac.il}{\texttt{{\small panzo@campus.technion.ac.il}}}}

\date{\today}

\begin{document}

\maketitle

\begin{abstract}
The conformal Skorokhod embedding problem (CSEP) is a planar variant of the classical problem where the solution is now a simply connected domain $D\subset\mathbb{C}$ whose exit time embeds a given probability distribution $\mu$ by projecting the stopped Brownian motion onto the real axis. In this paper we explore two new research directions for the CSEP by proving general bounds on the principal Dirichlet eigenvalue of a solution domain in terms of the corresponding $\mu$ and by proposing related extremal problems. Moreover, we give a new and nontrivial example of an extremal domain $\mathbb{U}$ that attains the lowest possible principal Dirichlet eigenvalue over all domains solving the CSEP for the uniform distribution on $[-1,1]$. Remarkably, the boundary of $\mathbb{U}$ is related to the Grim Reaper translating solution to the curve shortening flow in the plane. The novel tool used in the proof of the sharp lower bound is a precise relationship between the widths of the orthogonal projections of a simply connected planar domain and the support of its harmonic measure that we develop in the paper. The upper bound relies on spectral bounds for the torsion function which have recently appeared in the literature.
\end{abstract}

\begin{keywords}
conformal Skorokhod embedding; principal eigenvalue; harmonic measure; torsion function; Grim Reaper curve; catenary of equal resistance.
\end{keywords}

\begin{subjclass}
Primary 60G40, 60J65; Secondary 30C20, 30C70, 30C85.
\end{subjclass}


\section{Introduction}

Let $W=(W_t:t\geq 0)$ be a complex Brownian motion starting at $0$, and for any open set $D\subset\mathbb{C}$ containing $0$, let $\tau_D$ denote the first exit time of $W$ from $D$. Moreover, let $\mu$ be a probability distribution with zero mean and finite nonzero variance. The \emph{conformal Skorokhod embedding problem} (CSEP) was introduced by Gross in \cite{Gross2019} where the author shows that there exists a simply connected domain $D\subset\mathbb{C}$ containing $0$ such that
\begin{equation}\label{eq:CSEP}
\left\{
\begin{array}{cc}
\Real W_{\tau_D}\sim\mu\\
\\
\mathbb{E}_0[\tau_D]<\infty.
\end{array}
\right.
\end{equation}

This result was recently generalized in \cite{Baudabra-Markowsky2019}. There the authors showed that if $\mu$ is a probability distribution with zero mean and finite nonzero $p$-th moment for some $1<p<\infty$, then there exists a simply connected domain $D\subset\mathbb{C}$ containing $0$ such that $ \Real W_{\tau_D}\sim\mu$ and $\mathbb{E}_0[\tau_D^{p/2}]<\infty$. They also give conditions on the domain which ensure that a solution domain is unique.

The original Skorokhod embedding problem (SEP) asks the following: Given a standard Brownian motion $(B_t:t\geq 0)$ and a probability distribution $\mu$ with zero mean and finite variance, find a stopping time $T$ such that $B_T\sim\mu$ and $\mathbb{E}_0\left[T\right]<\infty$. It was first posed and solved by Skorokhod in 1961 (see \cite{Skorokhod-1965} for an English translation of the original paper) and since then a veritable zoo of varied and interesting solutions have appeared, see \cite{Jan2004} for a thorough survey. The similarity between these problems is clear, and as Gross points out, his solution resembles Root's barrier hitting solution \cite{Root1969} with additional randomness, that of the other Brownian motion in the second coordinate. 

Gross's paper is more than just an existence result and his method gives a relatively explicit construction of a domain $D$ which solves \eqref{eq:CSEP}. In fact, his example of a bounded domain which solves the CSEP for the uniform distribution on $[-1,1]$ is what inspired the present paper as it complements an unbounded domain $\mathbb{U}$ known to the authors which achieves the same objective. We discovered $\mathbb{U}$ through a naive ansatz that uses the conformal invariance of Brownian motion and the Cauchy-Riemann equations to prescribe a conformal deformation of the upper half-plane $\mathbb{H}$ (whose harmonic measure is Cauchy distributed) such that the orthogonal projection of the harmonic measure of the deformed domain has uniform distribution on $[-1,1]$. Remarkably, the boundary of $\mathbb{U}$ is related to the so-called \emph{Grim Reaper curve} which is a translating solution to the curve shortening flow in the plane, see Remark \ref{rem:Grim_Reaper}. 

Besides bringing to light the new example of $\mathbb{U}$, the main goals of this paper are to obtain \textbf{sharp bounds} on the principal Dirichlet eigenvalue of a solution domain in terms of the corresponding $\mu$ and to \textbf{propose related extremal problems}. For a domain $D\subset\mathbb{C}$ which solves the CSEP for $\mu$, define the \emph{rate} of the solution by
\[
\lambda(D)=-\lim_{t\to\infty}\frac{1}{t}\log\mathbb{P}_0(\tau_D>t).
\]
The limit exists by Fekete's subadditivity lemma and is clearly nonnegative. Since $D$ is an open set containing $0$, the process $W$ will exit some small enough ball with positive radius centered at $0$ before it exits $D$ so it follows that $\lambda(D)$ is finite. Additionally, $2\,\lambda(D)$ is equal to the bottom of the spectrum of the semigroup generated by the Laplacian on $D$ with Dirichlet boundary conditions and is usually referred to as the \emph{principal Dirichlet eigenvalue} of $D$, see Section 3.1 of \cite{Sznitman}. An example which will prove useful later on is that of an infinite strip of width $w$ containing $0$ where a straightforward projection argument shows that the rate is half the principal eigenvalue of an interval of length $w$, namely $\frac{\pi^2}{2w^2}$. Two questions regarding the rate naturally present themselves:
\begin{enumerate}[label=\Alph*.]
\item Find upper and lower bounds on the rate (or principal Dirichlet eigenvalue) in terms of $\mu$. \label{q:1}
\item For a specific $\mu$, find extremal domains that attain the highest and lowest possible rate (or principal Dirichlet eigenvalue). \label{q:2}
\end{enumerate}

The \textbf{main results} of the paper are the following: We answer Question \ref{q:1} by giving an upper bound in terms of the variance of $\mu$ in Theorem \ref{thm:upper_bound} and by giving a sharp lower bound in terms of the width of the support of $\mu$ in Theorem \ref{thm:lower_bound}. The proof of Theorem \ref{thm:upper_bound} relies on recent spectral bounds for the torsion function which are described at the beginning of Section \ref{sec:Main_Result}. In order to prove Theorem \ref{thm:lower_bound}, we need a precise relationship between the width of the orthogonal projection of a simply connected planar domain and the width of the support of the orthogonal projection of its harmonic measure. This result, which is stated in Theorem \ref{thm:support_theorem} and proven in Section \ref{sec:support}, may be of independent interest. As far as the authors know, it hasn't appeared in the literature. Finally, we give a partial answer to Question \ref{q:2} in the case of the uniform distribution on $[-1,1]$, henceforth denoted by $\mathrm{U}[-1,1]$, by verifying in Theorem \ref{thm:uniform_curve} that $\mathbb{U}$ is a nontrivial minimal rate solution domain. This example shows that the lower bound in Theorem \ref{thm:lower_bound} is indeed sharp. 

About one month after these results first appeared on arXiv, Boudabra and Markowsky posted a preprint \cite{Baudabra-Markowsky2020a} that gives an interesting new method to construct minimal rate solution domains to the CSEP. We point out that the minimality of their solution domains is a consequence of our general lower bound given in Theorem \ref{thm:lower_bound} which itself relies crucially on the most technical result of our paper, that of Theorem \ref{thm:support_theorem}. Additionally, their result shows the importance of Theorem \ref{thm:lower_bound} by demonstrating that our lower bound is sharp in general and not just for $\mathrm{U}[-1,1]$ which was already known from Theorem \ref{thm:uniform_curve}. 

Another example of a minimal rate solution domain which is in some sense trivial can be found in \cite{Baudabra-Markowsky2019} where it was shown that the infinite strip $\{z\in\mathbb{C}:|\Imaginary z|<1\}$ and the domain bounded below by the parabola $y=\frac{1}{2}x^2-\frac{1}{2}$ both solve the CSEP for the hyperbolic secant density $\frac{1}{2}\sech\frac{\pi}{2}x$ on $\mathbb{R}$ with the infinite strip being Gross's solution, see Figure \ref{fig:parabola}. Since the parabola domain has infinite inradius, it follows that its principal eigenvalue is $0$ and hence it trivially has minimal rate. Finding a maximal rate solution to the CSEP remains an open problem and one wonders if Gross's construction might be the answer.

Question \ref{q:2} is in the same spirit as the classical SEP where different solutions often have their own extremal property. For instance, given a probability distribution $\mu$ with zero mean and finite variance, Root's embedding minimizes the variance of $T$ over all stopping times $T$ such that $B_T\sim\mu$ and $\mathbb{E}_0\left[T\right]<\infty$, while Rost's reversed barrier embedding maximizes the variance of $T$ over the same class of stopping times. We refer to \cite{extremal_Skorokhod} where extremal solutions to the classical SEP are studied in detail.

Extremal problems for Brownian motion apart from the SEP have also been a popular topic of study. A common theme is optimizing the principal Dirichlet eigenvalue or the maximum expected exit time of a domain taken over all starting points when various constraints are given, see \cite{Banuelos-Carrol1994, conditioned_extremal, Kawohl_Sweers, Banuelos-Carrol2011, Steinerberger_Hoskins, improved_Vogt} and references therein for some examples.  

The CSEP also has a connection to Walden and Ward's \emph{harmonic measure distributions} \cite{Walden_Ward} that is worth mentioning. Indeed, if the domain $D$ solves the CSEP for $\mu$, then the orthogonal projection of the harmonic measure of $D$ with pole at $0$ onto the real axis has distribution $\mu$. If instead we are given a distribution function $F:(0,\infty)\to [0,1]$, then one might ask whether there exists a domain $D\subset\mathbb{C}$ containing $0$ such that the circular projection of the harmonic measure of $D$ with pole at $0$ onto the positive real axis has distribution function $F$? This is a topic of current interest and the reader is directed to the recent survey \cite{Snipes_Ward} for more on this question. Perhaps an adaptation of Gross's construction will prove fruitful in this area.

\section{Main Results}\label{sec:Main_Result}

Our first main result is an upper bound on $\lambda(D)$ for $d$-dimensional Brownian motion in terms of the second moments of the components at time $\tau_D$. When applied to a domain $D\subset\mathbb{C}$ which solves the CSEP, this general result provides a partial answer to Question \ref{q:1} The proof involves a \emph{spectral upper bound for the torsion function} of Brownian motion. Recall that the torsion function of a domain $D$ is nothing but the expected exit time of Brownian motion from $D$ as a function of the starting point and is given by
\begin{equation*}
\begin{split}
u_D(x)&=\mathbb{E}_x\left[\tau_D\right],~x\in D\\
&=\begin{cases}-\frac{1}{2}\Delta u_D=1\\ 
u_D\in H_0^1(D).\end{cases}
\end{split}
\end{equation*}
The best constant in the spectral upper bound for the torsion function depends on the dimension $d$ and is defined by
\begin{equation}\label{eq:best_constant}
C_d=\sup\left\{ \lambda\left(D\right)\|u_D\|_\infty\colon D\subset\mathbb{R}^d~\text{is a domain with}~\lambda\left(D\right)>0\right\}.
\end{equation}
It was shown in \cite{vdB_Carroll} that $\|u_D\|_\infty<\infty$ if and only if $\lambda(D)>0$, so it follows from \eqref{eq:best_constant} that the spectral upper bound for the torsion function
\begin{equation}\label{eq:spectral_torsion}
\|u_D\|_\infty\leq C_d\,\lambda(D)^{-1}
\end{equation}
holds for all domains $D$.
 
While computing $C_d$ exactly for $d>1$ is still an open problem, bounds on $C_d$ for Brownian motion and related processes have been studied extensively via several techniques under various assumptions on $D$, see for instance \cite{Banuelos-Carrol1994,vdB_Carroll,Giorgi-Smits-2010,Vandenberg-2017a,Panzo2020}. The current best explicit upper bound in the Brownian case was derived by H. Vogt in \cite{Vogt} and states that 
\begin{equation}\label{eq:VogtBound}
C_d\leq\frac{d}{8}+\frac{1}{4}\sqrt{5\left(1+\frac{1}{4}\log 2\right)d}+1.
\end{equation}
See \cite{improved_Vogt} for a non-explicit improvement of \eqref{eq:VogtBound} along with new results on the $p$-torsion analogue of \eqref{eq:spectral_torsion} which involves the $p$-th moment of the exit time. 

With the spectral upper bound for the torsion function \eqref{eq:spectral_torsion} at our disposal, we can now state our first main result.
\begin{thm}\label{thm:upper_bound}
Suppose $D$ is an open subset of $\mathbb{R}^d$ which contains $0$ and let $W=(W_t:t\geq 0)$ be a $d$-dimensional Brownian motion starting at $0$ with $W_t=(W_t^{(1)},\dots,W_t^{(d)})$. Then for each $1\leq i\leq d$ we have
\[
\lambda(D)\leq C_d\,\mathbb{E}_0\left[\left(W^{(i)}_{\tau_D}\right)^2\right]^{-1}.
\]
In particular, if $D$ solves the CSEP \eqref{eq:CSEP} for a probability distribution $\mu$ with zero mean and finite variance, then 
\begin{equation}\label{eq:CSEP_bound}
\lambda(D)\leq \frac{C_2}{\Var\mu}\leq\frac{2.0379}{\Var\mu}.
\end{equation}
\end{thm}

\begin{rem}
Under the added condition that $D$ is convex and if we assume that the \emph{equilateral triangle conjecture} of \cite{Henrot} is true, then the inequality \eqref{eq:CSEP_bound} can be improved to
\[
\lambda(D)\leq \frac{4\pi^2}{27\Var\mu}\approx\frac{1.4622}{\Var\mu}.
\]
\end{rem}

\begin{cor}
If $D$ solves the CSEP for $\mathrm{U}[-1,1]$, then 
\[
\lambda(D)\leq 3\,C_2\leq 6.1136.
\] 
\end{cor}

Our next result is a lower bound on $\lambda(D)$ in terms of the width of the support of $\mu$ which is sharp as demonstrated by Theorem \ref{thm:uniform_curve} below. When paired with Theorem \ref{thm:upper_bound}, this provides a full answer to Question \ref{q:1} 
\begin{thm}\label{thm:lower_bound}
Suppose $D\subset\mathbb{C}$ solves the CSEP \eqref{eq:CSEP} for a probability distribution $\mu$. If $[\alpha,\beta]$ is the smallest interval containing the support of $\mu$, then
\begin{equation}\label{eq:lower_bound}
\lambda(D)\geq \frac{\pi^2}{2(\beta-\alpha)^2}.
\end{equation} 
\end{thm}

\begin{cor}\label{cor:uniform_curve}
If $D$ solves the CSEP for $\mathrm{U}[-1,1]$, then 
\[
\lambda(D)\geq \frac{\pi^2}{8}\approx 1.2337.
\] 
\end{cor}

The proof of Theorem \ref{thm:lower_bound} relies on Theorem \ref{thm:support_theorem} below which gives a precise relationship between the width of the orthogonal projection of a simply connected planar domain and the width of the support of the orthogonal projection of its harmonic measure. As far as the authors know, this hasn't appeared in the literature before. We point the reader to section \ref{sec:support} where this result is developed in detail.

Our final result partially answers Question \ref{q:2} by exhibiting a domain $\mathbb{U}$ which is a minimal rate solution to the CSEP for $\mathrm{U}[-1,1]$. Moreover, this shows that our general lower bound given in Theorem \ref{thm:lower_bound} is sharp. See Figure \ref{fig:shapes} for a comparison of $\mathbb{U}$ and Gross's solution domain to the CSEP for $\mathrm{U}[-1,1]$. 

\begin{figure}
\centering
\includegraphics[scale=.46]{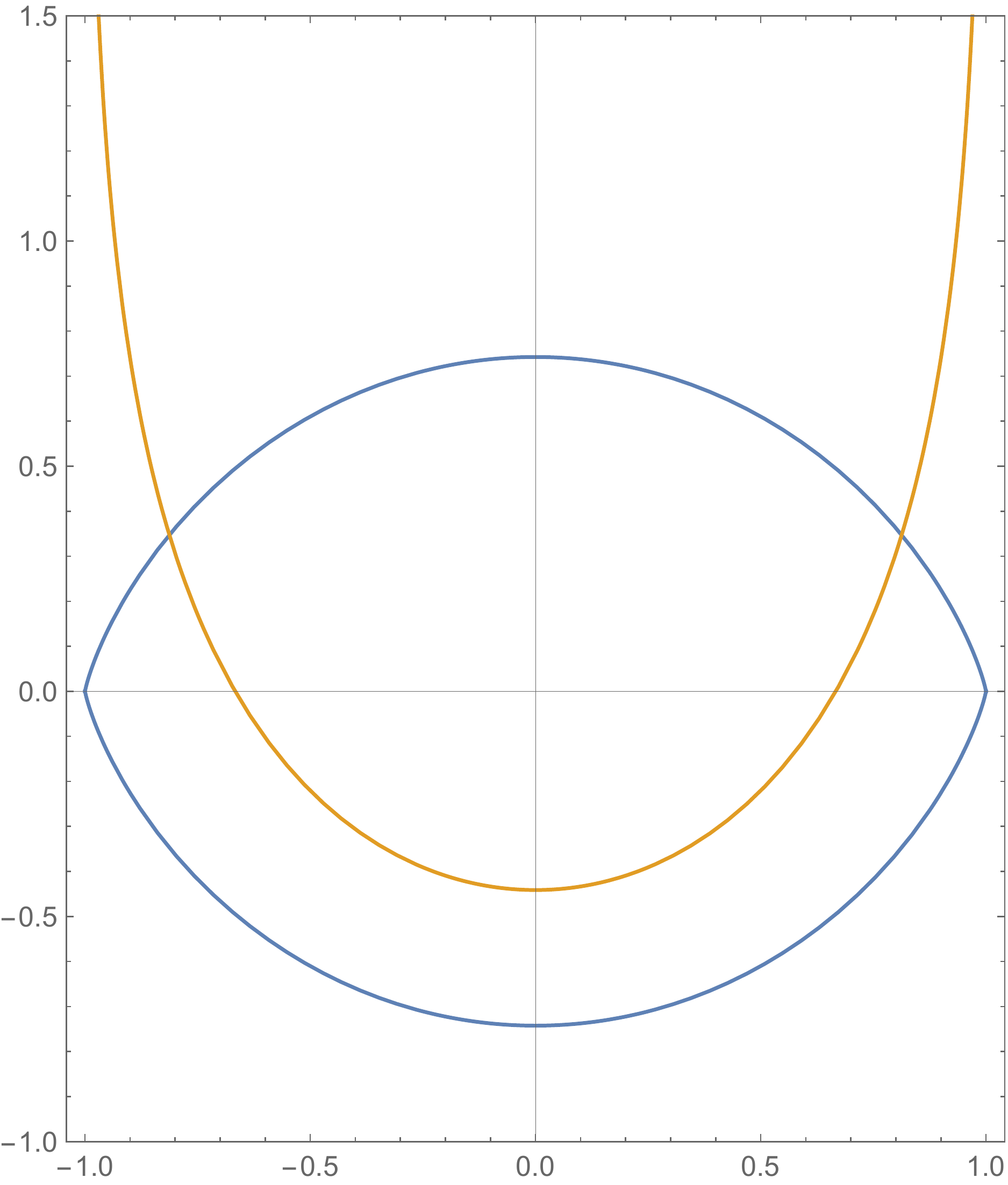}
\caption{$\mathbb{U}$ is the region above the U-shaped graph. Gross's solution to the CSEP for $\mathrm{U}[-1,1]$ is the eye-shaped region bounded by the closed curve.}
\label{fig:shapes}
\end{figure}

\begin{thm}\label{thm:uniform_curve}
Let 
\begin{equation}\label{eq:uniform_curve}
\mathbb{U}=\left\{ z\in\mathbb{C}:|\Real z|<1,\,\Imaginary z>h( \Real z)\right\}
\end{equation}
where 
\[
h(x)=-\frac{2}{\pi}\log\left(2\cos\frac{\pi}{2} x\right),~|x|<1.
\]
Then under $\mathbb{P}_0$ we have
\begin{equation*}
\Real W_{\tau_\mathbb{U}} \sim \mathrm{U}[-1,1].
\end{equation*}
Moreover, $\mathbb{U}$ is a minimal rate solution to the CSEP for $\mathrm{U}[-1,1]$. That is, if $D\subset\mathbb{C}$ is another solution to the CSEP \eqref{eq:CSEP} for $\mathrm{U}[-1,1]$, then $\lambda(\mathbb{U})\leq\lambda(D)$. 
\end{thm}

\begin{rem}\label{rem:Grim_Reaper}
The boundary of $\mathbb{U}$ is a translation and scaling of $y=-\log\cos x$ hence it belongs to the family of Grim Reaper curves which are the only self-similar translating solutions to the curve shortening flow in the plane, see \cite{Grim_Reaper}.
\end{rem}

\begin{rem}
The simply connected domain $\mathbb{U}$ is a rotation and scaling of a domain studied in \cite[Example 4]{Markowsky2011} where the author points out that its boundary curve has been referred to as the ``catenary of equal resistance". In that paper the expected exit time is computed but no mention is made of the distribution of the real or imaginary parts of the stopped Brownian motion.
\end{rem}

\begin{rem}
The family of scaled domains $a\mathbb{U}$ with $a>0$ gives minimal rate solutions to the CSEP for $\mathrm{U}[-a,a]$.
\end{rem}

\section{Orthogonal Projection of Harmonic Measure}\label{sec:support}

The goal of this section is to prove a relationship between the width of the projection of $D$ onto the real axis and the width of the support of $\Real W_{\tau_D}$. More specifically, suppose $D\subsetneq\mathbb{C}$ is a simply connected domain containing $0$ and let $\Real D$ denote its projection onto the real axis, that is, $\Real D=\{\Real z:z\in D\}$. Since $D$ is both open and connected and the projection map is both open and continuous, we know that $\Real D=(a,b)$ for some $a\in [-\infty,0)$ and $b\in (0,\infty]$.

The boundary $\partial D$ is nonpolar under the assumptions on $D$ so $\tau_D$ is almost surely finite. Define 
\[
\alpha=\sup\left\{x\in\mathbb{R}:\mathbb{P}_0\big(\Real W_{\tau_D}\in[x,\infty)\big)=1\right\}
\]
and
\[
\beta=\inf\left\{x\in\mathbb{R}:\mathbb{P}_0\big(\Real W_{\tau_D}\in(-\infty,x]\big)=1\right\}.
\]
It is clear that $\alpha\leq\beta$ by definition. Moreover, since $W_{\tau_D}\in \partial D$ almost surely, it follows that $[\alpha,\beta]\subset [a,b]$. The following theorem gives conditions under which the reverse inclusion also holds. While this may seem intuitively obvious at first glance, pathological examples of simply connected domains abound \cite{Pommerenke} so it is worthwhile writing a careful proof. Besides, Remark \ref{rem:alpha_beta} and Figure \ref{fig:alpha_beta} show that $[\alpha,\beta]\supset [a,b]$ doesn't hold in complete generality even for nice domains.

\begin{thm}\label{thm:support_theorem}
Suppose $D\subsetneq\mathbb{C}$ is a simply connected domain containing $0$ and let $a,\alpha,b,\beta$ be defined as above. Then $a>-\infty$ implies $\alpha=a$ and $b<\infty$ implies $\beta=b$. Moreover, if $\mathbb{E}_0[\tau_D]<\infty$, then $\alpha=a$ and $\beta=b$ regardless of whether $a$ and $b$ are finite.
\end{thm}
\begin{rem}\label{rem:alpha_beta}
Considering simply connected domains such as the interior of $\mathbb{U}^c$ (translated so that it contains $0$) show that an additional condition is needed for $\alpha=a$ and $\beta=b$ to hold when $a$ and $b$ are infinite, see Figure \ref{fig:alpha_beta}. The sufficient condition $\mathbb{E}_0[\tau_D]<\infty$ is one that happens to fit well with the CSEP. 
\end{rem}

\begin{figure}
\centering
\begin{minipage}{0.48\textwidth}
\centering
\includegraphics[width=\textwidth]{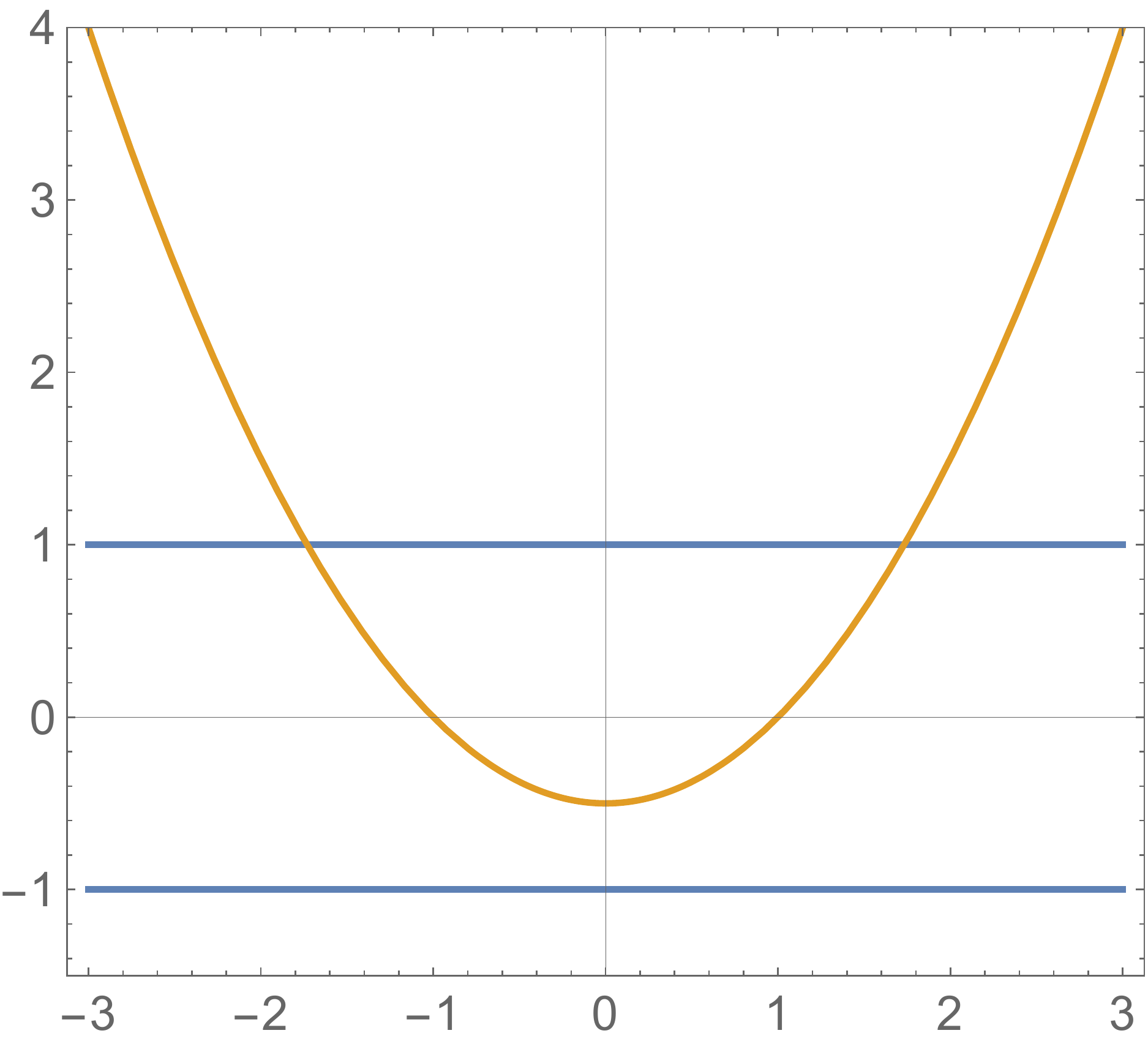}
\caption{The infinite strip and the region above the parabola both embed the hyperbolic secant density.}
\label{fig:parabola}
\end{minipage}
\hfill
\begin{minipage}{0.48\textwidth}
\centering
\includegraphics[width=\textwidth]{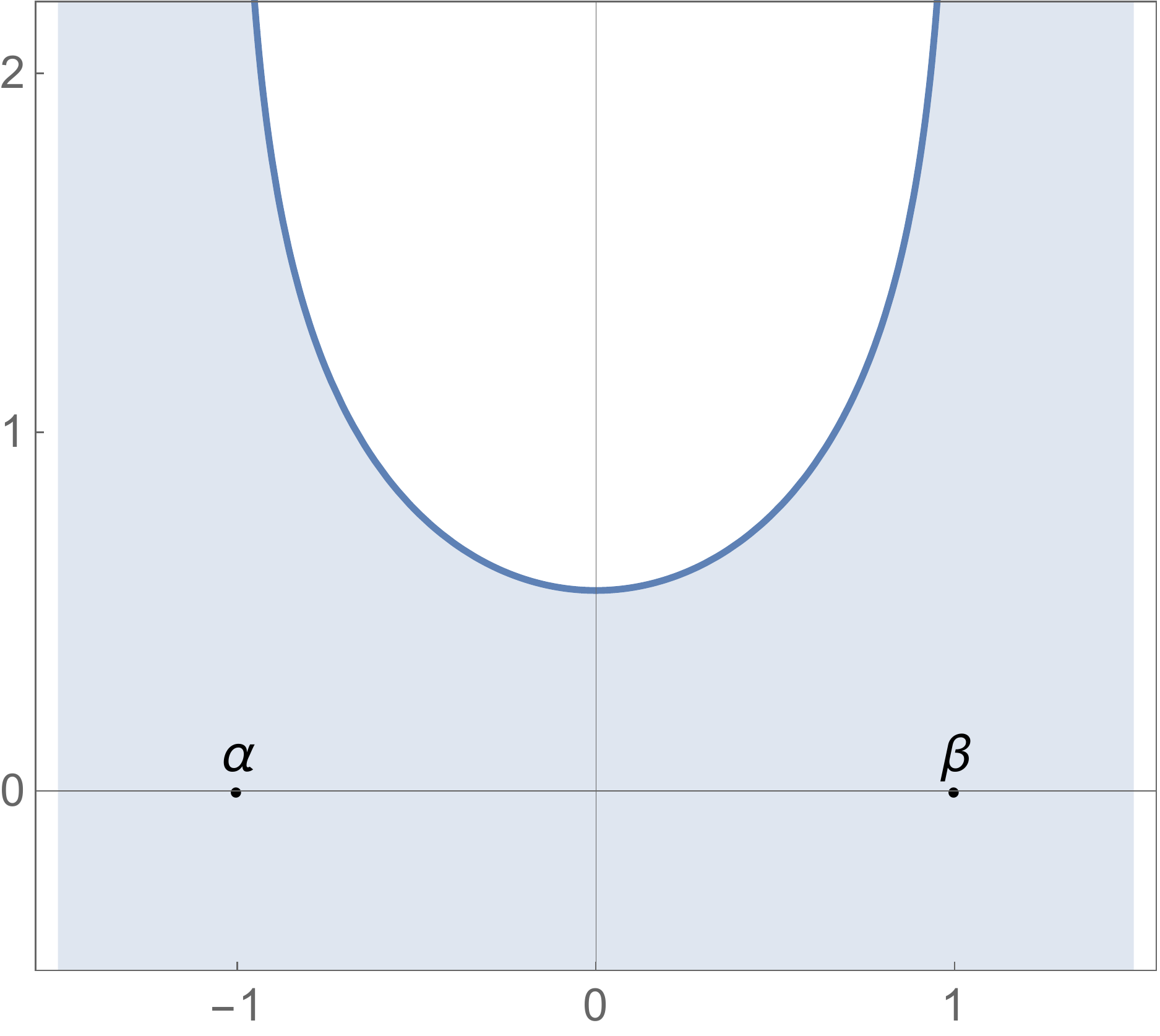}
\caption{Considering the interior of $\mathbb{U}^c+i$ shows that $\alpha=a$ and $\beta=b$ may not hold when $a$ and $b$ are infinite.}
\label{fig:alpha_beta}
\end{minipage}
\end{figure}

The proof of Theorem \ref{thm:support_theorem}, which we postpone until the end of the section, relies on the next lemma which states that there is positive probability of $W$ exiting $D$ through any neighborhood of any boundary point. In what follows, we denote by $B_r(x)$ and $\overline{B}_r(x)$ the open and closed disk, respectively, of radius $r>0$ centered at $x$.
\begin{lem}\label{lem:support_lemma}
Suppose $D\subset\mathbb{C}$ is a simply connected domain and let $x\in D$ and $\xi\in\partial D$. Then for all $\epsilon>0$ we have 
\[
\mathbb{P}_x\left(W_{\tau_D}\in \overline{B}_\epsilon(\xi)\cap \partial D\right)>0.
\]
\end{lem}

\begin{proof}
Fix $\epsilon>0$ and pick $y\in B_\epsilon(\xi)\cap D$. Since $D$ is a domain, $x$ and $y$ can be connected by a polygonal path in $D$. This path is a compact subset of $D$ so it has a positive distance from $\partial D$. Hence the Harnack inequality implies
\begin{equation}\label{eq:Harnack}
\mathbb{P}_x\left(W_{\tau_D}\in \overline{B}_\epsilon(\xi)\cap \partial D\right)>0~\text{ iff }~\mathbb{P}_y\left(W_{\tau_D}\in \overline{B}_\epsilon(\xi)\cap \partial D\right)>0
\end{equation}
so we can focus on proving that the latter inequality holds. Since $W$ has continuous paths, it will hit $\partial D$ before it hits the interior of $D^c$, hence the latter probability appearing in \eqref{eq:Harnack} is equal to
\begin{equation}\label{eq:complement}
\mathbb{P}_y\left(W_{\tau_D}\in \overline{B}_\epsilon(\xi)\cap D^c\right).
\end{equation}
Next, we can bound \eqref{eq:complement} from below by 
\begin{equation}\label{eq:truncate}
\mathbb{P}_y\left(W_{\tau_{B_\epsilon(\xi)\cap D}}\in \overline{B}_\epsilon(\xi)\cap D^c\right)
\end{equation}
since \eqref{eq:truncate} excludes paths that leave $B_\epsilon(\xi)$ and return to hit $\overline{B}_\epsilon(\xi)\cap D^c$. By letting $K=\overline{B}_\epsilon(\xi)\cap D^c$, we can write \eqref{eq:truncate} as 
\[
\mathbb{P}_y\left(W_{\tau_{B_\epsilon(\xi)\setminus K}}\in K\right).
\]

Define the conformal transformation $f(w)=-\frac{|y-\xi|}{y-\xi}(w-\xi)$ which maps $\xi$ to $0$ and $y$ to $-|y-\xi|\in (-\epsilon,0)$. For a set $E\subset \mathbb{C}$, let $E^*$ denote the circular projection of $E$. That is, $E^*=\{|w|:w\in E\}$. Now we can use conformal invariance of Brownian motion and the version of Beurling's projection theorem from \cite[Theorem 1]{Oksendal} to get
\begin{align*}
\mathbb{P}_y\left(W_{\tau_{B_\epsilon(\xi)\setminus K}}\in K\right)&=\mathbb{P}_{f(y)}\left(W_{\tau_{B_\epsilon(0)\setminus f(K)}}\in f(K)\right)\\
&\geq \mathbb{P}_{f(y)}\left(W_{\tau_{B_\epsilon(0)\setminus f(K)^*}}\in f(K)^*\right).
\end{align*}

In order to produce a meaningful estimate, $f(K)^*$ must contain a proper interval. We claim $f(K)^*=[0,\epsilon]$. To see that this is true, consider the connected component of $K$ that contains $\xi$, call it $E$. Clearly $E$ contains some point $z$ such that $|z-\xi|=\epsilon$, for otherwise $E$ would be a bounded connected component of $D^c$ which contradicts $D$ being simply connected. This implies both $0$ and $\epsilon$ are elements of $f(E)^*$. Since $E$ is connected and both $f$ and circular projection are continuous, we know that $f(E)^*$ is connected. Hence 
\[
[0,\epsilon]\supset f(K)^*\supset f(E)^*\supset [0,\epsilon]
\]
and the claim follows.

Finally, we can use $w\mapsto \sqrt{w/\epsilon}$ to conformally map $B_\epsilon(0)\setminus [0,\epsilon]$ onto the upper half-disk $\mathbb{D}\cap\mathbb{H}$, thereby sending $f(y)$ to $\sqrt{f(y)/\epsilon}$ and the boundary set $[0,\epsilon]$ to $[-1,1]$. Using the above inequalities along with the explicit formula for the harmonic measure of $\mathbb{D}\cap\mathbb{H}$ \cite[Table 4.1]{Ransford} while noting that $\sqrt{f(y)/\epsilon}$ is purely imaginary, we can write
\begin{align*}
\mathbb{P}_y\left(W_{\tau_D}\in \overline{B}_\epsilon(\xi)\cap \partial D\right)&\geq 1-\frac{2}{\pi}\arg\left(\frac{1+\sqrt{f(y)/\epsilon}}{1-\sqrt{f(y)/\epsilon}}\right)\\
&=1-\frac{2}{\pi}\arctan\left(\frac{2\sqrt{|f(y)|/\epsilon}}{1-|f(y)|/\epsilon}\right)\\
&>0
\end{align*}
where the last inequality follows from $|f(y)|=|y-\xi|<\epsilon$. In conjunction with \eqref{eq:Harnack}, this proves the lemma. 
\end{proof}

With Lemma \ref{lem:support_lemma} in hand, we can now give a proof of Theorem \ref{thm:support_theorem}.
\begin{proof}[Proof of Theorem \ref{thm:support_theorem}]
Suppose $b<\infty$ and let $\epsilon>0$. Then there exists $x\in D$ such that $|\Real x-b|<\epsilon$. Hence $B_\epsilon(x)\cap D^c$ is nonempty, for otherwise we could increase $b$. It follows that $B_\epsilon(x)$ must contain some $\xi\in\partial D$ and that $|\Real\xi-b|<2\epsilon$. 

By Lemma \ref{lem:support_lemma}, we know that $\mathbb{P}_0\left(W_{\tau_D}\in \overline{B}_\epsilon(\xi)\cap \partial D\right)>0$, from which it follows that $\mathbb{P}_0\left(\Real W_{\tau_D}> b-3\epsilon\right)>0$. Hence $\mathbb{P}_0\left(\Real W_{\tau_D}\in (-\infty, b-3\epsilon]\right)<1$. This implies that $b-3\epsilon<\beta\leq b$ so we can conclude that $\beta=b$. The proof that $\alpha=a$ follows similarly.

Now suppose that $\mathbb{E}_0[\tau_D]<\infty$. If $b<\infty$, we have already shown that $\beta=b$, so assume $b=\infty$. We will show that $\beta<\infty$ leads to a contradiction. 

A result of Burkholder \cite[Equation 3.13]{Burkholder} shows that 
\begin{equation}\label{eq:Burkholder}
\mathbb{E}_0[\tau_D]<\infty~\text{ implies }~\mathbb{E}_x[\tau_D]<\infty~\text{ for all }~x\in D.
\end{equation}
By hypothesis, there exists $x\in D$ such that $\Real x>\beta$. Note that $\Real \xi\leq \beta$ for all $\xi\in\partial D$, for otherwise we could use Lemma \ref{lem:support_lemma} to increase $\beta$. Since $W_{\tau_D}\in\partial D$ almost surely, this implies 
\[
\mathbb{E}_x\left[\tau_D\right]\geq\mathbb{E}_x\big[\inf\{t\geq 0:\Real W_t\leq\beta\}\big]=\infty.
\]
In light of \eqref{eq:Burkholder}, this leads to the desired contradiction. The proof that $\alpha=a$ follows similarly.

\end{proof}

\section{Proofs of Main Results}\label{sec:Proofs}

\subsection{Proof of Theorem \ref{thm:upper_bound}}

\begin{proof}[Proof of Theorem \ref{thm:upper_bound}]
If $\lambda(D)=0$ there is nothing to prove, so assume that $\lambda(D)>0$. In this case we can use \eqref{eq:best_constant} to write 
\begin{align}\label{eq:Vogt}
\lambda(D)&\leq C_d\, \|u_D\|_\infty^{-1}\nonumber \\
&\leq C_d\,\mathbb{E}_0\left[\tau_D\right]^{-1}.
\end{align}
Since the components of $W$ are all independent, each $W^{(i)}$ is a Brownian motion in the natural filtration of $W$, with respect to which $\tau_D$ is a stopping time. In particular, optional stopping applied to the martingale $(W^{(i)}_t)^2-t$ implies that
\[
\mathbb{E}_0\left[\left(W^{(i)}_{\tau_D\wedge n}\right)^2\right]=\mathbb{E}_0\left[\tau_D\wedge n\right]
\]
for each $1\leq i\leq d$ and $n\in\mathbb{N}$. Letting $n\to\infty$ in the above equality while using Fatou's lemma on the left-hand side and monotone convergence on the right-hand side leads to 
\begin{equation}\label{eq:Davis}
\mathbb{E}_0\left[\left(W^{(i)}_{\tau_D}\right)^2\right]\leq\mathbb{E}_0\left[\tau_D\right]
\end{equation}
for each $1\leq i\leq d$. Combining \eqref{eq:Vogt} and \eqref{eq:Davis} proves the theorem. The numerical estimates follow from \cite[Table 1]{improved_Vogt}.
\end{proof}

\subsection{Proof of Theorem \ref{thm:lower_bound}}

\begin{proof}[Proof of Theorem \ref{thm:lower_bound}]
If $\beta-\alpha=\infty$, then \eqref{eq:lower_bound} gives the trivial lower bound of $0$ so there is nothing to prove in this case. Hence we can assume that $\alpha>-\infty$ and $\beta<\infty$. Since $D$ is a solution domain to the CSEP \eqref{eq:CSEP}, we know that $\mathbb{E}_0\left[\tau_D\right]<\infty$. By Theorem \ref{thm:support_theorem}, this implies that $\{\Real z:z\in D\}=(\alpha,\beta)$. Hence $D$ is contained in the infinite strip $\mathbb{S}_{\alpha,\beta}=\{z\in\mathbb{C}:\alpha<\Real z<\beta\}$. Since $\lambda(\mathbb{S}_{\alpha,\beta})=\frac{\pi^2}{2(\beta-\alpha)^2}$, the result follows by domain monotonicity of Dirichlet Laplacian eigenvalues \cite[Lemma 3.1.1]{Sznitman}.
\end{proof}

\subsection{Proof of Theorem \ref{thm:uniform_curve}}\label{sec:uniform_curve}

The first step in proving Theorem \ref{thm:uniform_curve} is to calculate the rate of $\mathbb{U}$. This will be essential in establishing that it is indeed a minimal rate solution to the CSEP for $\mathrm{U}[-1,1]$. We do this in the following lemma.
\begin{lem}\label{lem:uniform_rate}
\[
\lambda(\mathbb{U})=\frac{\pi^2}{8}
\]
\end{lem}

\begin{proof}
Since $\mathbb{U}$ is contained in the infinite strip $\{z\in\mathbb{C}:|\Real z|<1\}$, it follows from domain monotonicity of Dirichlet Laplacian eigenvalues that $\lambda(\mathbb{U})\geq \frac{\pi^2}{8}$. To get an upper bound, notice from \eqref{eq:uniform_curve} that the rectangles 
\[
R_n:=\left\{z\in\mathbb{C}:|\Real z|<1-\frac{1}{n},\,h\left(1-\frac{1}{n}\right)<\Imaginary z<h\left(1-\frac{1}{n}\right)+n\right\},~n\geq 2
\]
are all contained in $\mathbb{U}$ and have width $2(1-\frac{1}{n})$ and height $n$. Separation of variables can be used to calculate $\lambda(R_n)$ explicitly, which, along with domain monotonicity, allows us to write 
\[
\lambda(R_n)=\frac{\pi^2}{2}\left(\frac{1}{4\left(1-\frac{1}{n}\right)^2}+\frac{1}{n^2}\right)\geq \lambda(\mathbb{U}),~n\geq 2.
\]
Letting $n\to\infty$ in the above inequality completes the proof.
\end{proof}

Now that we know $\lambda(\mathbb{U})$, it remains to show that $\mathbb{U}$ is actually a solution to the CSEP for $\mathrm{U}[-1,1]$. The finite width of $\mathbb{U}$ implies $\mathbb{E}_0[\tau_\mathbb{U}]<\infty$ and we can show that $\Real W_{\tau_\mathbb{U}}\sim\mathrm{U}[-1,1]$ by way of an explicit conformal map from $\mathbb{U}$ onto the upper half-plane $\mathbb{H}$ while exploiting conformal invariance of Brownian motion.
\begin{proof}[Proof of Theorem \ref{thm:uniform_curve}]
Consider the holomorphic function 
\[
f(z)=2i\,e^{-\frac{\pi}{2}iz}-i.
\]
Since $z\mapsto e^z$ is injective on the infinite strip $\{z\in\mathbb{C}:|\Imaginary z|<\frac{\pi}{2}\}$ and since $\mathbb{U}$ is contained in the infinite strip $\{z\in\mathbb{C}:|\Real z|<1\}$, it follows that $f$ is injective on $\mathbb{U}$. 

From \eqref{eq:uniform_curve}, we know that $\partial\mathbb{U}=\left\{x+i\,h(x):x\in (-1,1)\right\}$. It will be convenient to foliate $\mathbb{U}$ by vertical translates of $\partial\mathbb{U}$. More specifically, for each $y\geq 0$, define $h_y=\left\{x+i(h(x)+y):x\in (-1,1)\right\}$. Then $\partial\mathbb{U}=h_0$ and $\mathbb{U}=\bigcup_{y>0} h_y$. Since $f$ is injective on $\mathbb{U}$, we can foliate $f(\mathbb{U})$ by the images of $h_y$ under $f$. Towards this end, we compute
\begin{align}
f(h_y)&=\left\{2i \exp\left(-\frac{\pi}{2}i x-\log\left(2\cos\frac{\pi}{2} x\right)+\frac{\pi}{2}y\right)-i:x\in(-1,1)\right\}\nonumber \\
&=\left\{2i e^{\frac{\pi}{2}y}\frac{\cos\frac{\pi}{2}x-i\sin\frac{\pi}{2}x}{2\cos\frac{\pi}{2}x}-i:x\in(-1,1)\right\}\nonumber \\
&=\left\{e^{\frac{\pi}{2}y}\tan\frac{\pi}{2}x+i\left(e^{\frac{\pi}{2}y}-1\right)-i:x\in(-1,1)\right\}\label{eq:boundary} \\
&=\left\{z\in\mathbb{C}:\Imaginary z=e^{\frac{\pi}{2}y}-1\right\}.\nonumber
\end{align}
This shows that $f(\mathbb{U})=\bigcup_{y>0} f(h_y)=\mathbb{H}$, hence $f$ maps $\mathbb{U}$ conformally onto $\mathbb{H}$.

Now for $0\leq x<1$, the conformal invariance of Brownian motion and the fact that $\partial \mathbb{U}$ is given by the graph of the function $h$ imply that 
\begin{align*}
\mathbb{P}_0\left(0\leq\Real W_{\tau_\mathbb{U}}\leq x\right)&=\mathbb{P}_{f(0)}\left(\Real f\big(i\,h(0)\big)\leq\Real W_{\tau_\mathbb{H}}\leq \Real f\big(x+i\,h(x)\big)\right)\\
&=\mathbb{P}_i\left(0\leq\Real W_{\tau_\mathbb{H}}\leq \tan\frac{\pi}{2}x\right)\\
&=\frac{1}{2}x.
\end{align*}
In the middle equation we used \eqref{eq:boundary} with $y=0$ and in the last equation we used the fact that the law of $\Real W_{\tau_\mathbb{H}}$ under $\mathbb{P}_i$ is just the harmonic measure of the upper half-plane with pole at $i$ which is well known to have the standard Cauchy distribution \cite[Table 4.1]{Ransford}. Together with symmetry considerations, this shows that $\Real W_{\tau_\mathbb{U}}\sim \mathrm{U}[-1,1]$ under $\mathbb{P}_0$. In light of Corollary \ref{cor:uniform_curve} and Lemma \ref{lem:uniform_rate}, it follows that $\mathbb{U}$ is a minimal rate solution to the CSEP for $\mathrm{U}[-1,1]$. 
\end{proof}

\begin{acknowledgments}
The authors would like to thank Matthew Badger for bringing their attention to Walden and Ward's harmonic measure distributions as well as Robert Neel for pointing out that the boundary of $\mathbb{U}$ belongs to the Grim Reaper curve family. We also thank Renan Gross and Greg Markowsky for useful comments on an earlier draft, Rodrigo Ba\~{n}uelos and Luke Rogers for helpful discussions, and the anonymous reviewer for their careful reading and valuable remarks.
\end{acknowledgments}

\bibliography{GenThm}
\bibliographystyle{alphabbrev}

\end{document}